\documentclass[10pt, article]{amsart}
\usepackage{geometry} 
\usepackage{ae} 
\usepackage[T1]{fontenc}
\usepackage[cp1250]{inputenc}
\usepackage{amsmath}
\usepackage{amssymb, amsfonts,amscd,verbatim}
\usepackage{MnSymbol}
\usepackage{mathrsfs} 
\usepackage{xypic}

\usepackage[normalem]{ulem}
\usepackage{hyperref}
\usepackage{indentfirst}
\usepackage{latexsym}
\input xy
\xyoption{all}

\usepackage{amsmath}    

\theoremstyle{plain}
\newtheorem{Pocz}{Poczatek}[section]
\newtheorem{Proposition}[Pocz]{Proposition}

\newtheorem{Theorem}[Pocz]{Theorem}

\newtheorem{Lemma}[Pocz]{Lemma}
\newtheorem{Observation}[Pocz]{Observation}

\theoremstyle{definition}

\theoremstyle{remark}

\numberwithin{equation}{section}
\author{Thomas ~ Weighill}
\address{University of Tennessee, Knoxville, USA}
\email{tweighil@vols.utk.edu}

\title{A characterization of Gromov hyperbolicity via quasigeodesic subspaces}

\date{ \today
}

\begin{document}


\begin{abstract}
By a geodesic subspace of a metric space $X$ we mean a subset $A$ of $X$ such that any two points in $A$ can be connected by a geodesic in $A$. It is easy to check that a geodesic metric space $X$ is an $\mathbb{R}$-tree (that is, a $0$-hyperbolic space in the sense of Gromov) if and only if the union of any two intersecting geodesic subspaces is again a geodesic subspace. In this paper, we prove an analogous characterization of general Gromov hyperbolic spaces, where we replace geodesic subspaces by quasigeodesic subspaces.
\end{abstract}

\maketitle

\section{Introduction}
Recall that a geodesic metric space $X$ is called \textbf{$\delta$-hyperbolic} (in the sense of Gromov) if for any three points $a$, $b$ and $c$ in $X$ and geodesics $[a ,b]$, $[b, c]$ and $[a, c]$ between them, the geodesic $[b,c]$ is contained in a $\delta$-neighbourhood of $[a, b] \cup [a, c]$. An example of a Gromov hyperbolic space is a metric tree with the path-length metric: such trees are $0$-hyperbolic spaces. More generally, geodesic metric spaces which are $0$-hyperbolic are called \textbf{$\mathbb{R}$-trees}. 

By a \textbf{geodesic subspace} of a metric space $X$ we mean a subset $A$ of $X$ such that any two points in $A$ are joined by a geodesic in $A$. The following observation is easy to verify.

\begin{Observation}\label{obs}
A metric space $X$ is an $\mathbb{R}$-tree if and only if for any geodesic subspaces $A$ and $B$ with $A \cap B \neq \varnothing$, $A\cup B$ is a geodesic subspace.
\end{Observation}
In this paper, we give a characterization of Gromov hyperbolic spaces similar to the one for $\mathbb{R}$-trees given in the above observation.

\section{Main Result}
Let $X$ be a metric space. We say that a subset $A \subseteq X$ is a \textbf{$(\lambda, C)$-quasigeodesic} subspace if for every pair of points $a$ and $a'$ in $A$ there is a map $f: I \rightarrow A$ from an interval $I = [0,v]$ of the real line, such that $f(0) = a$, $f(v) = a'$ and for every pair of points $x$, $x'$ in $X$,
$$
\lambda^{-1}d(x,x') - C \leq d(f(x), f(x')) \leq \lambda d(x, x') + C.
$$ 
In other words, any pair of points in $A$ are joined by a $(\lambda, C)$-quasigeodesic in $A$ in the sense of \cite{Roe}. In particular, a geodesic subspace is the same as a $(1,0)$-quasigeodesic subspace.

\begin{Lemma}\label{1quasi}
Let $X$ be a $\delta$-hyperbolic geodesic metric space. Then for every $\lambda > 0$ and $C > 0$ there is a $C'>0$ such that for any $(\lambda, C)$-quasigeodesic $\gamma$, there is a $(1, C')$-quasigeodesic whose image is contained in that of $\gamma$ and which has the same endpoints. 
\end{Lemma}
\begin{proof}
Let $\gamma$ be a $(\lambda, C)$-quasigeodesic. Since $X$ is hyperbolic, there is an $R>0$ depending only on $\lambda$ and $C$ such that the image of $\gamma$ has Hausdorff distance at most $R$ from a geodesic with the same endpoints (see e.g.~\cite{Roe}). Let $\gamma'$ be this geodesic, and define a new quasigeodesic $\gamma''$ with the same domain as $\gamma'$ by sending $x$ to a point in the image of $\gamma$ which is within $R$ of $\gamma'(x)$, choosing the endpoints to be the same. Then $\gamma''$ is a $(1, 2R)$-quasigeodesic whose image is contained in the image of $\gamma$.
\end{proof}

We are now ready to prove the main result. Note the similarity of conditions (1) and (2) with the condition in Observation~\ref{obs}.

\begin{Theorem} \label{main}
The following are equivalent for geodesic metric space $X$:
\begin{itemize}
\item[(1)] for every $C > 0$ there exists a $C'$ such that if $A,B \subseteq X$ are $(1, C)$-quasigeodesic subspaaces and $A \cap B \neq \varnothing$, then $A \cup B$ is a $(1, C')$-quasigeodesic subspace.
\item[(2)] for every $\lambda, C > 0$ there exists $\lambda', C' > 0$ such that if $A,B \subseteq X$ are $(\lambda, C)$-quasigeodesic subspaces and $A \cap B \neq \varnothing$, then $A \cup B$ is a $(\lambda', C')$-quasigeodesic subspace.
\item[(3)] $X$ is $\delta$-hyperbolic for some $\delta$.
\end{itemize}
\end{Theorem}
\begin{proof}
(3) $\implies$ (1): Let $A$ and $B$ be $(1, C)$-quasigeodesic subspaces of $X$, and let $w \in A \cap B$. Since there is a $(1, C)$-quasigeodesic joining any two points in $A$ and any two points in $B$, it is enough to consider $a \in A$ and $b \in B$. Since $X$ is $\delta$-hyperbolic, there is a point on the geodesic from $a$ to $b$ which is at most $\delta$ from the geodesic from $w$ to $a$ and the geodesic from $w$ to $b$. Moreover, there is an $R > 0$ depending only on $C$ such that a $(1, C)$-quasigeodesic from $w$ to $a$ (resp. $b$) has Hausdorff distance at most $R$ from a geodesic from $w$ to $a$ (resp. $b$). It follows that there is a point $c$ on the geodesic from $a$ to $b$ and points $a'$ and $b'$ on $(1, C)$-quasigeodesics from $w$ to $a$ in $A$ and $w$ to $b$ in $B$ respectively such that $d(a', c) \leq R + \delta$ and $d(b', c) \leq R + \delta$. Let $\gamma_a: [0, t_a] \to A$ and $\gamma_b: [0, t_b] \to B$ be the quasigeodesics in question. For convenience, we suppose $\gamma_a$ to be from $a$ to $w$ (so that $\gamma_a(0) = a$) while $\gamma_b$ is from $w$ to $b$ (so that $\gamma_b(t_b) = b$). Let $\gamma_a(s_a) = a'$ and $\gamma_b(s_b) = b'$. We will build a new quasigeodesic $\gamma: [0, t] \rightarrow X$ from $a$ to $b$, where $t = s_a + t_b - s_b$, whose image is contained in $A \cup B$. We define it as follows:
\begin{align*}
\gamma(x) = 
\begin{cases}
\gamma_a(x) & 0 \leq x \leq s_a \\
\gamma_b(x - s_a + s_b) & s_a \leq x  \leq t
\end{cases}
\end{align*}
If $x \in [0, s_a]$ and $y \in [s_a, t]$, then we have
$$
d(\gamma(x), \gamma(y)) \leq d(\gamma(x), \gamma(s_a)) + 2(R + \delta) + d(\gamma(s_a), y) \leq d(x, y) + 2C + 2(R + \delta)
$$
On the other hand,
$$
d(\gamma(x), \gamma(y)) \geq d(a,b) - d(a, \gamma(x)) - d(b, \gamma(y)) \geq d(a,b) - d(0, x) -  d(y, t) - 2C
$$
and 
$$
d(a,b) \geq d(a, a') - (R + \delta) + d(b', b) - (R + \delta) \geq d(0, s_a) - 2(R + \delta) + d(s_a, t) - 2C
$$
gives
\begin{align*}
d(\gamma(x), \gamma(y)) & \geq  d(0, s_a) - 2(R + \delta) + d(s_a, t)  - d(0, x) -  d(y, t) - 4C \\ & = d(x, y) - 4C - 2(R+\delta).
\end{align*}
so $\gamma$ is a $(1, 4C + 2R + 2\delta))$-quasigeodesic.

(3) \& (1) $\implies$ (2): By Lemma \ref{1quasi}, every $(\lambda, C)$-quasigeodesic subspace is also $(1, C')$-quasigeodesic for some $C'$ depending on $\lambda$ and $C$. 

(2) $\implies$ (3): Any connected union of four geodesic segments is, by hypothesis, a $(\lambda, C)$-quasigeodesic subspace for some fixed $\lambda$ and $C$. Consider a geodesic triangle $abc$ in $X$, and let $x$ be the point on $[a,b]$ (the geodesic from $a$ to $b$) which is furthest from $[b,c]$ and $[a,c]$. If $d(a,x) \leq C + 1$ or $d(x,b) \leq C + 1$, then $x$ is at most $C + 1$ away from one of the other two sides. Suppose $d(a,x) > C + 1$ and $d(b,x) > C + 1$. Let $x_a$ and $x_b$ be points on $[a,x]$ and $[x,b]$ respectively such that $d(x_a, x) = d(x_b, x) = (C + 1)/2$. The union $Y = [a,x_a]\cup [b,x_b] \cup [b,c] \cup [a,c]$ is a connected union of four geodesic subspaces, so it must be a $(\lambda, C)$-quasigeodesic subspace. In particular, there is a $(\lambda, C)$-quasigeodesic $\gamma$ joining $x_a$ and $x_b$ whose image is contained in $Y$. We claim that the image of $\gamma$ intersects $[a,c] \cup [b,c]$ at some point $z$. If this is the case, then
$$
d(x_a, z) \leq \lambda t + C \leq \lambda (\lambda (d(x_a, x_b) + C)) + C = \lambda^2(2C+1) + C
$$
where $t$ is the length of the interval which is the domain of $\gamma$. To prove the claim, note that for $x \in \gamma^{-1}([a, x_a])$ and $y \in \gamma^{-1}([b_, x_b])$, we have
$$
\lambda d(x,y) + C \geq d(\gamma(x), \gamma(y)) \geq C + 1 \implies d(x,y) \geq \lambda^{-1} > 0 
$$
so that by connectedness, $\gamma^{-1}([a, x_a])$ and $\gamma^{-1}([b_, x_b])$ cannot cover the domain of $\gamma$. Thus $X$ is $\delta$-hyperbolic, for $\delta = \lambda^2(2C+1) + C + 1$.
\end{proof}

\end{document}